\documentclass{article}
\usepackage{latexsym}
\usepackage[cp850]{inputenc}
\usepackage[sumlimits,intlimits,namelimits,centertags]{amsmath}
\usepackage{amsfonts}
\usepackage{amsmath}
\usepackage{rotating}
\usepackage{theorem}

\newtheorem{theorem}{Theorem}
\newtheorem{lemma}{Lemma}

\newtheorem{proposition}{Proposition}

\newenvironment{proof}{\textbf{Proof}}{\hfill $\Box$ \vspace{10pt}\\}

\newcommand{\R}{\mathbb R}

\newcommand{\N}{\mathbb N}

\newcommand{\T}{\mathbb T}
\newcommand{\C}{\mathbb C}
\newcommand{\Z}{\mathbb Z}
\newcommand{\Q}{\mathbb Q}

\title {Continuous functions with universally divergent Fourier series on small subsets of the circle}

\author{J\"urgen M\"uller}

\begin{document}

\maketitle
\begin{abstract}
It is shown that quasi all continuous functions on the unit circle have the property that, for many small subsets $E$ of the circle, the partial sums of their Fourier series considered as functions restricted to $E$ exhibit certain universality properties.

\end{abstract}

AMS-classification: 42A63,42A65,54E52


\section{Introduction and results}

Since Du Bois Reymond's work it is known that Fourier series of continuous functions may diverge in some points of the unit circle $\T=\R/2\pi\Z$. Of course, in view of Carleson's theorem, this can only happen on a set of vanishing Lebesgue measure. 

Actually, nowadays divergence is seen as generic phenomenon:

Let $(X,d)$ be a complete metric space. A property is said to be satisfied by quasi all $x \in X$ if it is satisfied on a residual set in $X$ (that is, on a set containing a dense $G_\delta$ set).

We consider, for compact sets $E \subset \T$,
\[
C(E):=\{f: E \to \C, \, f \mbox{ continuous}\}, \qquad ||f||_E:=\max_{t \in E } |f(t)|
\]
as well as $||\cdot||:=||\cdot||_\T$. 

Then quasi all $f \in (C(\T), ||\cdot||)$ enjoy the property that their Fourier series diverge in quasi all points of $\T$ (see, for example, Theorem 1.1 in Kahane's expository paper \cite{Kah} which is a rich source of related results). 

The question arises to what extend divergence of $(s_n f)_n$, where $s_n f$ denotes the $n$-th partial sum of the Fourier series of $f$, may take place for such functions. 
Our first result shows that for countable $E \subset \T$ the divergence may be "maximal" on $E$:

\begin{theorem} \label{theorem1}
For each countable set $E \subset \T$ quasi all $f \in C(\T)$ enjoy the property that for every function $h: E \to \C$ there is a subsequence of $(s_n f)$ converging pointwise to $h$ on $E$. 
\end{theorem}

Proofs are postponed to Section 2.\\

The question arises which uncountable sets $E$ show a "pointwise universality" property as in Theorem \ref{theorem1} for some $f \in C(\T)$. Of course, possible limit functions are resticted to the first Baire class. It is well-known that trigonometric series (and even Taylor series of functions holomorphic in the unit disk) exist with the property that every function of the first Baire class is a pointwise limit of a subsequence of the series on $\T$ (see, e.g. \cite{Kah}, Corollary 2.3, and \cite{Nes}). However, in the case of Fourier series of continuous functions Carleson's theorem shows that $E$ must have vanishing Lebesgue measure and, moreover, it turns out that universal divergence sets     
also have to be topologicaly small in the sense of being of first category. This is in some contrast to the well-known fact that every set $E$ of vanishing measure is a set of (unbounded) divergence for some $f \in C(\T)$ (see, e. g. \cite{Kat}, Theorem II 3.4). 

In fact, from Carleson's theorem we can easily obtain that $f$ has special attraction as limit point not only almost everywhere but also quasi everywhere:\\

{\it If $f\in C(\T)$ and $(n_j)$ are so that $(s_{n_j} f)_j$ converges pointwise to some function $h$ on a set $A\subset \T$, then $h$ differs from $f$ at most on a set of first category.\\}
 
This follows immediately from 

\begin{proposition}\label{prop}
If $f\in C(\T)$ and $(s_{n_j} f)_j$ is a subsequence of $(s_n f)_n$ then there is a dense $G_\delta$-subset $E$ of $\T$ so that for all $t\in E$ a subsequence of $(s_{n_j} f(t))$ tends to $f(t)$.
\end{proposition}

We may also ask about uniform convergence of subsequences of $(s_n f)$ on subsets $E$ of the unit circle. If $\{s_n f_{|E}: n \in \N\}$ is dense in $C(E)$ we say that $(s_n f)$ is uniformly universal on $E$. 
Denoting by ${\cal K}(\T)$ the (complete) metric space of all compact, nonempty subsets of $\T$ equipped with the Hausdorff metric, we have

\begin{theorem} \label{theorem2}
Quasi all continuous functions on $\T$ enjoy that property that $(s_n f)$ is uniformly universal on quasi all sets $E \in{\cal K}(\T)$. 
\end{theorem}

It can be shown (see \cite{Koe}) that quasi all elements of ${\cal K}(\T)$ are perfect sets. In particular, there are perfect -- and thus uncountable -- sets $E$ so that $(s_n f)$ is uniformly universal on $E$ for some $f \in C(\T)$.
Again, it would be interesting to know for which sets $E\in {\cal K}(\T)$ such functions exist. Due to Carleson's theorem, they necessarily have to have vanishing Lebesgue measure, and thus are in particular nowhere dense in $\T$. 
The proofs in Section 2 show that a "localized version" of Theorem \ref{theorem2} is true, i.e. for each compact set $E_0 \subset \T$ quasi all $f \in C(\T)$ are so that $(s_n f)$ is uniformly universal on $E$ for quasi all $E \in {\cal{K}}(E_0)$. \\

Since $C(\T)$ is dense in $L_p(\T)$ for all $1 \le p < \infty$, and since $s_n:L_p(\T) \to C(\T)$ is continuous for all $n$, the proofs in Section 2 show that the Theorems \ref{theorem1} and \ref{theorem2} also hold with $L_p(\T)$ instead of $C(\T)$. In view of Kolmogorov's result on the divergence everwhere of Fourier series of (quasi-all) $L_1$-functions (see \cite{Kah}, Theorem 1.5), it is reasonable to suspect that in the $L_1$-case not only sets $E$ of vanishing Lebesgue measure are permitted. However, compact sets $E$ such that $(s_n f)$ is uniformly universal on $E$ for some integrable $f$ still have to be nowhere dense (see \cite{Kah}, p. 157).

\section{Proofs}  

The key tool is

\begin{lemma} \label{le1}
For all finite sets $E\subset \T$ 
\[
\big\{ f\in C(\T): \{s_n f: n \in \N\} \mbox{ dense in } \C^E (=C(E)) \big\}
\]
is a dense $G_\delta$-set in $C(\T)$.
\end{lemma}

\begin{proof}

1. We show that for all $g\in C(\T), c\in\C, \varepsilon > 0, \Lambda \subset\N, |\Lambda| = \infty$ there is $f \in C(\T), n\in\Lambda$, so that
\[
\|f - g\| < \varepsilon, \qquad (s_n f)(0) = c.
\]
Indeed: The Weierstra{\ss} approximation theorem implies the existence of a trigonometric polynomial $q$ with $\|g-q\| < \varepsilon/2$. 

For quasi all $f\in C(\T)$ the sequence $\big( s_n f (0)\big)_{n\in\Lambda}$ is unbounded (this is a well-known consequence of the Banach-Steinhaus theorem in the case $\Lambda = \N$ and may be obtained in the same way for arbitrary $\Lambda$, since the $L_1$-norms of the $n$-th Dirichlet kernel tend to infinity). Therefore, there is a $h\in C(\T)$ with $\|h\| < \varepsilon/2$ and $\big|(s_n h)(0)\big| > \big| c - q(0)\big|$ for some $n\in\Lambda, \, n\geq \text{deg}\,(q)$.
Then 
$$
\widetilde{h} := \frac{c-q(0)}{s_n h(0)} \cdot h
$$ 
satisfies
\[
s_n \widetilde{h} (0) = c - q(0) \qquad \text{and} \qquad \|\widetilde{h}\| \leq \|h\| < \varepsilon/2
\]
and thus for $f := \widetilde{h} + q$ we obtain (since $n \geq \text{deq }(q)$)
\[
\|f-g\| < \varepsilon \qquad \text{and} \qquad
(s_n f)(0) = (s_n \widetilde{h}) (0) + q(0) = c.
\]
According to the Universality Criterion (see, for example, \cite{GE}), applied to the sequence
$T_n: C(\T) \to \C$ with $T_n f := (s_n f) (0)$ $(n\in\Lambda, f\in C(\T))$ the set
\[
\big\{ f\in C(\T): \{(s_n f)(0): n\in \Lambda\} \mbox{ dense in } \C \, (= \C^{\{0\}})\big\}
\]
is $G_\delta$-dense in $C(\T)$.\\

2. We prove the assertion by induction on $N = |E|$. 

For $N = 1$ the result follows from part 1 of the proof (without loss of generality we may suppose $E = \{0\}$). 

Let $E \subset \T$ with $|E| = N+1$ be given. Again, we suppose that $0\in E$. We write
$E = F\cup \{0\}$ with $|F| = N$. Then by induction hypothesis, $\{s_n(f): n \in \N\}$ is dense in $\C^F$ for quasi all $f\in C(\T)$. 

The universality criterion, now applied to $T_n :  C(\T) \to \C^E$, $T_n f := s_n f_{|E}, (n\in\N, f\in C(\T)$), shows that it is sufficient (and necessary) to guarantee that for all 
$g\in C(\T), \varepsilon > 0, h: E\to \C$ there exist $f\in C(\T)$ and $n\in\N$ with
\[
\|f-g\| < \varepsilon \qquad\text{and} \qquad \|s_n f - h\|_E < \varepsilon.
\]
Due to the Weierstra{\ss} approximation theorem, we may suppose that $g$ is a trigonometric polynomial. 

Let $u\in C(\T)$ be so that $\|u\| < \varepsilon/2$, and let $\Lambda\subset\N$ satisfy 
$|\Lambda| = \infty$ and
\[
\big| (s_n u) (t) - \big( h(t) - g(t)\big)\big| < \varepsilon/3 \qquad (n\in\Lambda, t\in F).
\]
By 1. there are $v\in C(\T), \|v\| < \varepsilon/2$ and $\Lambda' \subset\Lambda, |\Lambda'| = \infty$ so that
\[
\big| (s_n v) (0) - \big( h(0) - g(0)\big)\big| < \varepsilon/3 \qquad (n\in \Lambda').
\]
Now, consider $\varphi\in C(\T)$ with $0 \leq \varphi \leq 1$ and
\[
\varphi_{|U} \equiv 1 \; , \qquad \varphi_{|V} \equiv 0
\]
for neighborhoods $U$ of $F$ and $V$ of 0 in $\T$. 
The Localization Principle (see e. g. \cite{Kat}, Theorem II 2.4) implies that
\[
\| s_n (u\varphi) - s_n u\|_F < \varepsilon/3\;, \qquad \big| s_n (u\varphi)(0)\big| < \varepsilon/3
\]
and
\[
\big| s_n \big(v (1-\varphi)\big) (0) - (s_n v)(0)\big| < \varepsilon/3, \quad
\big\| \big(s_n (1-\varphi) v\big)\big\|_F < \varepsilon/3
\]
for $n$ sufficiently large. 
We define 
\[
f := u\varphi + v(1-\varphi) + g.
\]
Then $\|f-g\|\leq \|u\| + \|v\| < \varepsilon$ and for $n\in \Lambda'$ with $n\geq \text{deg} (g)$ we have (since $s_n g = g$)
\begin{eqnarray*}
\| s_n f - h\|_F & \leq & \big\| s_n \big( v (1-\varphi)\big)\big\| + \| s_n (u\varphi)+ g-h \|_F \\
& \leq & \frac{\varepsilon}{3} + \| s_n (u\varphi) - s_n u\|_F + \big\| s_n u - (h-g)\big\|_F < \varepsilon
\end{eqnarray*}
and similarly 
\[
|(s_n f)(0) - h(0)|  
 \leq \big| (s_n (u\varphi)(0)\big| + \big| \big( s_n (v(1-\varphi))\big) (0) + g(0) - h(0)\big| < \varepsilon.
\]
\end{proof} 

\begin{proof} {\bf of Theorem \ref{theorem1}}

Let $(E_j)_{j \in \N}$ be an exhaustion of $E$ consisting of finite sets. Then by Lemma 1 and Baire's theorem
quasi all $f \in C(\T)$ are so that the partial sums $s_n f$ are dense in $\C^{E_j}$ for all $j \in \N$. If $f$ is of that kind, and if $h:E \to \C$ is given, then for all $j$ there is an integer $n_j > n_{j-1}$ with $||s_{n_j}f-h||_{E_j} < 1/j$. Thus $(s_{n_j}f)$ converges pointwise to $h$ on $E$. 
\end{proof}

\begin {proof} {\bf of Proposition \ref{prop}}

We have 
\begin {eqnarray*}
\{ t\in \T: \exists (j_\ell)_\ell: (s_{n_{j_\ell}} f) (t) \to f(t) \}
& = & \bigcap\limits_{m,k\in\N} \, \bigcup\limits_{j\geq m} 
\Big\{ t: \big| s_{n_j} f(t) - f(t)\big| < \frac{1}{k}\Big\} \\
& =: &\bigcap\limits_{m,k\in\N} {\cal O}_{m,k}
\end{eqnarray*}
with ${\cal O}_{m,k}$ open in $\T$. According to Baire's theorem, it is sufficient to show that 
${\cal O}_{m,k}$ is dense for all $(m,k)$. 

Suppose that there is some $(m,k)$ with ${\cal O}_{m,k}$ not dense in $\T$. Then, for all
$t\in U$ where $U$ is the complement of the closure of ${\cal O}_{m,k}$, 
\[
\big| s_{n_j} f(t) - f(t)\big|\geq \frac{1}{k} \qquad  (j\geq m).
\]
This contradicts the fact that $s_n f(t) \to f(t)$ on a dense subset of $\T$ (following from Carleson's theorem).
\end{proof}

Theorem \ref{theorem2} is an immediate consequence of Lemma \ref{le1} and the following Lemma (applied with $X=C(\T)$, $I=\N$, and $(T_n)=(s_n)$). 
\begin{lemma} \label{lemma2}
Let $X$ be a complete metric space, and let $(T_\alpha)_{\alpha \in I}$ be a family of functions $T_\alpha:X \to C(\T)$. If, for all finite $E \subset \T$, the set $\{T_\alpha(x)_{|E}: \alpha \in I\}$ is dense in $\C^E$ for quasi all $x \in X$, then $\{T_\alpha(x)_{|E}: \alpha \in I\}$ is dense in $C(E)$ for quasi all $x \in X$ and quasi all $E \in {\cal K}(\T)$.  
\end{lemma}
\begin {proof} 

1. Let $x\in X$ be fixed. Then
\[
{\cal E}_x := \big\{ E \in {\cal K}(\T): \{T_\alpha(x)_{|E}: \alpha \in I\} \mbox{ dense in } C(E)\big\}
\]
is a $G_\delta$-set in ${\cal K}(\T)$. 

Indeed: Let $Q$ denote the set of all trigonometric polynomials with coefficients from $\Q + i\Q$. 
Then $Q$ is countable and dense in $\big( C(E),\|\cdot\|_E\big)$ according to Tietze's extension theorem and the Weierstra{\ss} approximation theorem. Now
\[
{\cal{E}}_x = \bigcap\limits_{k\in\N \atop q\in Q}\, \bigcup\limits_{\alpha\in I}
\Big\{ E\in {\cal K}(\T): \| T_\alpha(x) - q\|_E < \frac{1}{k}\Big\}.
\]
Since $T_\alpha(x) - q$ is uniformly continuous on $\T$, we see that
$$
\{E\in {\cal K}(\T): \| T_\alpha(x) - q\|_E < \frac{1}{k}\}
$$
is open in ${\cal K}(\T)$, and this shows that ${\cal E}_x$ is a $G_\delta$-set in ${\cal K}(\T)$.\\

2. We fix a countable dense subset $M$ of $\T$. Let ${\cal E}$ denote the family of all finite subsets of $M$. Then ${\cal E}$ is countable and dense in ${\cal K}(\T)$. According to Baire's theorem,
\[
 \bigcap\limits_{E\in{\cal E}} \{ x\in X: \{T_\alpha(x)_{|E}: \alpha \in I\} \mbox{ dense in }\C^E \}
\]
is a residual subset of $X$. For all $x$ in this set we have ${\cal E} \subset {\cal E}_x$ and thus the first part of proof shows that ${\cal E}_x$ is $G_\delta$-dense in ${\cal K}(\T)$.
\end{proof}

Universit\"at Trier, FB IV, Mathematik, D-54286 Trier, Germany

\end{document}